\documentclass[a4paper,12pt,reqno]{amsart}

\usepackage[
backend=biber,
style=alphabetic,
sorting=nyt
]{biblatex}

\usepackage{amsfonts}
\usepackage{amsmath}
\usepackage{mathrsfs}
\usepackage{amssymb}
\usepackage{amsthm}
\usepackage{amscd}
\usepackage{latexsym}
\usepackage{amstext}
\usepackage{amsxtra}
\usepackage[utf8]{inputenc}
\usepackage[english]{babel}

\usepackage{geometry}
\usepackage{color}
\usepackage{paralist}

\usepackage[colorinlistoftodos]{todonotes}
\usepackage[all]{xy}
\usepackage{tikz}
\usepackage{tikz-cd}

\usepackage{enumerate}
\usepackage{multirow}
\usepackage{wasysym}
\usepackage{latexsym} 
\usepackage{comment}
\usepackage{colonequals}

\usepackage[
        draft=false,
        colorlinks, citecolor=darkgreen,
        pdfauthor={L.Fassina,G Pirola},
        pdftitle={},
        linktocpage        
]{hyperref}
\hypersetup{citecolor=blue,linktocpage}

\newtheorem{theorem}{Theorem}[section]
\newtheorem*{theorem*}{Theorem}

\newtheorem{corollary}[theorem]{Corollary}

\newtheorem{remark}[theorem]{Remark}
\newtheorem{definition}[theorem]{Definition}

\newcommand{\nc}{\newcommand} 
\nc{\cH}{{\mathcal H}}
\nc{\cA}{{\mathcal A}}
\nc{\cG}{{\mathcal G}}
\nc{\cC}{{\mathcal C}}
\nc{\cD}{{\mathcal D}}
\nc{\cO}{{\mathcal O}}
\nc{\cI}{{\mathcal I}}
\nc{\cB}{{\mathcal B}}
\nc{\cY}{{\mathcal Y}}
\nc{\cK}{{\mathcal K}} 
\nc{\cX}{{\mathcal X}}
\nc{\cS}{{\mathcal S}}
\nc{\cE}{{\mathcal E}}
\nc{\cF}{{\mathcal F}}
\nc{\cZ}{{\mathcal Z}}
\nc{\cQ}{{\mathcal Q}}
\nc{\cN}{{\mathcal N}}
\nc{\cP}{{\mathcal P}}
\nc{\cL}{{\mathcal L}}
\nc{\cM}{{\mathcal M}}
\nc{\cT}{{\mathcal T}}
\nc{\cW}{{\mathcal W}}
\nc{\cU}{{\mathcal U}}
\nc{\cJ}{{\mathcal J}}
\nc{\cV}{{\mathcal V}}
\nc{\bH}{{\mathbb H}}
\nc{\bA}{{\mathbb A}}
\nc{\bG}{{\mathbb G}}
\nc{\bC}{{\mathbb C}}
\nc{\bO}{{\mathbb O}}
\nc{\bI}{{\mathbb I}}

\nc{\bB}{{\mathbb B}}
\nc{\bY}{{\mathbb Y}}
\nc{\bK}{{\mathbb K}} 
\nc{\bX}{{\mathbb X}}
\nc{\bS}{{\mathbb S}}
\nc{\bE}{{\mathbb E}}
\nc{\bF}{{\mathbb F}}
\nc{\bZ}{{\mathbb Z}}
\nc{\bQ}{{\mathbb Q}}
\nc{\bN}{{\mathbb N}}
\nc{\bP}{{\mathbb P}}
\nc{\bL}{{\mathbb L}}
\nc{\bM}{{\mathbb M}}
\nc{\bT}{{\mathbb T}}
\nc{\bW}{{\mathbb W}}
\nc{\bU}{{\mathbb U}}
\nc{\bD}{{\mathbb D}}
\nc{\bJ}{{\mathbb J}}
\nc{\bV}{{\mathbb V}}
\nc{\bbZ}{{\mathbb Z}}
\nc{\bR}{{\mathbb R}}
\nc{\fr}{{\rightarrow}}
\nc{\co}{{\nabla}}
\newcommand{\debar}{{\bar\partial}}

\nc{\cu}{{\barline{\nabla}}}

\nc{\OO}{\mathcal{O}}
\nc{\PP}{\mathbb{P}}

\nc{\fA}{{\mathfrak{A}}}
\nc{\fB}{{\mathfrak{B}}}
\nc{\fC}{{\mathfrak{C}}}
\nc{\fD}{{\mathfrak{D}}}
\nc{\fE}{{\mathfrak{E}}}
\nc{\fF}{{\mathfrak{F}}}

\begin{document}


\title{On the Rank of Normal Functions on Pointed Curves}
\date{\today}

\author{Lorenzo Fassina}
\address{Dipartimento di Matematica,
	Universit\`a degli Studi di Pavia,
	Via Ferrata, 5
	I-27100 Pavia, Italy}
\email{lorenzo.fassina02@universitadipavia.it}




\begin{abstract}
We study a natural class of normal functions on the moduli space of pointed curves and determine their rank at every point. We prove that the rank is maximal outside the corresponding stratum of Abelian differentials, while along this stratum it drops by exactly one. As a consequence, we describe the local geometry of the associated double ramification locus.
\end{abstract}

\maketitle


\section{Introduction}

Normal functions play a fundamental role in the study of algebraic cycles via the Abel--Jacobi map. In this context, Griffiths \cite{infinitesimal} developed their infinitesimal study, which captures subtle geometric information on how algebraic cycles vary in families. Further developments by Green \cite{Green} and Voisin \cite{Voisin-une-remarque} have clarified the deep Hodge-theoretic nature of these objects.

A central role in this theory is played by the \textit{Griffiths infinitesimal invariant} which measures the first-order variation of a normal function and provides an obstruction to be locally constant.

The aim of this paper is to study, in the setting of pointed curves, the
loci where a natural class of normal functions
constructed from relative divisors are locally constant. Let us introduce the setting. Throughout the paper we work in the complex-analytic category. Let
\[
    \mu \colon Y \longrightarrow \mathcal{M}_{g,n}
\]
be an étale analytic chart of the moduli stack of smooth \(n\)-pointed curves of genus \(g\geq 2\), where \(Y\) is a connected complex manifold, and let
\begin{equation}\label{definition of family}
    \pi\colon (\mathcal{C},p_1,\ldots,p_n)\longrightarrow Y
\end{equation}
be the corresponding family. Thus, \(\pi\) is a smooth proper holomorphic family of projective curves and $p_i\colon Y\longrightarrow\mathcal{C}$ for $i=1,\ldots,n$ 
are pairwise disjoint holomorphic sections. For \(y\in Y\), we write
\[
    C_y:=\pi^{-1}(y),
    \qquad
    \mu(y)=[C_y;p_1(y),\ldots,p_n(y)].
\]

For every \(d\in\mathbb Z\), we denote by
\[
    \operatorname{Pic}^d(\pi)\longrightarrow Y
\]
the degree-\(d\) component of the relative Picard bundle associated with \(\pi\); its fibre over \(y\in Y\) is naturally identified with $\operatorname{Pic}^d(C_y)$.
In particular,
\[
    \operatorname{Pic}^0(\pi)\simeq \mathcal{J}
\]
is the relative Jacobian of the family. An holomorphic section $\nu$ of $\operatorname{Pic}^0(\pi)$ is called \emph{normal function}.

Let $\mathbf{m}:=(m_1,\dots,m_n) \in \mathbb{Z}^n$ be a  $n$-tuple of non-zero integers such that $m_1\leq\cdots\leq m_n$ and $d =\sum_{i=1}^n m_i$ and let
\begin{equation}\label{def-psi_m}
    \psi_\mathbf{m}:Y\longrightarrow \operatorname{Pic}^d(\pi),
\qquad
\psi_\mathbf{m}(y)=\mathcal O_{C_y}(D_y),
\end{equation}

be the holomorphic section of $\operatorname{Pic}^d(\pi)$ where $D_y$ is the divisor of degree $d$ defined by $\sum_{i=1}^n m_i p_i(y)$.

For a divisor $D$ on $C$, we denote by
\[
    \operatorname{Supp}(D)
\]
the set of points appearing in $D$ with non-zero coefficient.

The normal function studied in this paper is
\begin{equation}\label{eq:normal-function}
    \nu :=
    (2g-2)\psi_{\mathbf m}-d\kappa 
\end{equation}
where $\kappa\colon
    Y\longrightarrow\operatorname{Pic}^{2g-2}(\pi)$, defined by $\kappa(y):=\omega_{C_y}$, is the canonical section.
Equivalently, for every \(y\in Y\),
\[
    \nu(y)
    =
    \mathcal O_{C_y}\bigl((2g-2)D_y\bigr)
    \otimes
    \omega_{C_y}^{-d}.
\]

The object that will play a central role in the paper is the
space \(\mathcal N_y^\nu\) (see Definition~\ref{def:infinitesimal-local-constancy-space}),
consisting of those tangent directions along which the normal function
\(\nu\) is locally constant to first order at \(y\). Via the modular map, we
regard it as a subspace
\[
    \mathcal N_y^\nu
    \subset
    T_{[C_y;p_1(y),\ldots,p_n(y)]}\mathcal M_{g,n}.
\]
Since the fibres of
the relative Jacobian have dimension \(g\), one always has
\[
    \operatorname{codim} \mathcal N_y^\nu \leq g.
\]
The maximal possible codimension of \(\mathcal N_y^\nu\) is therefore \(g\). General results on the Betti map, due to Andr\'e--Corvaja--Zannier \cite{Andre-Corvaja-Zannier} and Gao \cite{Gao-Betti}, imply in this setting that this maximal value is attained at a general point, while Gao--Zhang \cite{Gao-Zhang-rank} developed a general framework for the corresponding degeneracy loci of normal functions. Our aim is more precise: we determine \(\operatorname{codim}\mathcal N_y^\nu\) at every point and identify exactly the locus where the maximal value \(g\) fails to occur.

Remarkably, the only source of such a degeneration is provided by the
geometry of strata of Abelian differentials. These strata are classical objects which have
been extensively studied from several points of view. Their connected
components were classified by Kontsevich--Zorich
\cite{Kontsevich}, while their geometry and
compactifications have been investigated, among others, by
Farkas--Pandharipande \cite{farkas-panda} and
Bainbridge--Chen--Gendron--Grushevsky--M\"oller \cite{BCGGM}.
When $d \neq 0$ set
\[
    a_i:=\frac{(2g-2)m_i}{d},
    \qquad i=1,\ldots,n.
\]
Assume that the \(a_i\) are positive integers and consider the locus
\[
    \mathcal P(a_1,\ldots,a_n)
    :=
    \left\{
    (C;p_1,\ldots,p_n)\in\mathcal M_{g,n}
    \ \middle|\
    \omega_C
    \simeq
    \mathcal O_C\!\left(\sum_{i=1}^n a_i p_i\right)
    \right\}.
\]
This is the projectivized stratum of Abelian differentials, viewed inside
\(\mathcal M_{g,n}\), and it has codimension \(g-1\).

The relevance of this locus to our problem is already visible directly from
the definition of \(\nu\). Indeed, along
\(\mathcal P(a_1,\ldots,a_n)\) one has
\[
    (2g-2)D
    \sim
    dK_C,
\]
and hence the normal function \(\nu\) vanishes identically. Therefore its
first-order variation vanishes along every tangent direction to the stratum,
so that
\[
    T_{[C;p_1,\ldots,p_n]}
    \mathcal P(a_1,\ldots,a_n)
    \subseteq
    \mathcal N_y^\nu.
\]
Since \(\mathcal P(a_1,\ldots,a_n)\) has codimension \(g-1\), this immediately
forces
\[
    \operatorname{codim} \mathcal N_y^\nu \leq g-1
\]
along the stratum.

This explains geometrically why the expected codimension \(g\) may fail:
on the Abelian stratum there is an entire family of infinitesimal
deformations along which the normal function is forced to remain constant.
The natural question is whether further degeneracies may occur, either
along this locus or elsewhere. Our main theorem shows that this does not
happen: the above geometric obstruction accounts for the whole
codimension drop.


\begin{theorem}\label{thm:max-rank-nu}
    Let $\mathbf{m} = (m_1,\dots,m_n) \in \mathbb{Z}^n$ be a $n$-tuple of non-zero integers with $\sum_{i = 1}^{n}m_i = d$ and  $m_1\leq\dots\leq m_n$. Let $\psi_\mathbf{m}$ as in \eqref{def-psi_m} and $\nu$ be the associated normal function \eqref{eq:normal-function}
    $$
        \nu=(2g-2)\psi_\mathbf{m}-d\kappa.
    $$
    For a point $y \in Y$ the spaces \(\mathcal N_y^\nu\) have the following behaviour.

\begin{enumerate}
    \item Suppose that $d\neq 0$, that $a_i = \frac{(2g-2)m_i}{d}$ are positive integers for $i = 1,\dots, n$ and that
    \[
    [C_y;p_1(y),\ldots,p_n(y)]
    \in
    \mathcal P(a_1,\ldots,a_n),
    \]
    then
    \[
    \operatorname{codim}_{T_{[C_y;p_i(y)]}\mathcal M_{g,n}}\mathcal N_y^\nu =  g-1.
    \]
    \item Otherwise,
    \[
    \operatorname{codim}_{T_{[C_y;p_i(y)]}\mathcal M_{g,n}}\mathcal N_y^\nu=g.
    \]

\end{enumerate}
\end{theorem}


Thus Theorem~\ref{thm:max-rank-nu} gives a pointwise refinement of the generic maximality results above: the maximal value $\operatorname{codim}\mathcal N_y^\nu=g$ holds at every point outside the Abelian stratum, while along the stratum the codimension drops by exactly one. The proof builds on the infinitesimal approach developed in \cite{fassina_pirola}, using Griffiths' infinitesimal invariant and higher Schiffer deformations.

The theorem also gives a direct consequence for the associated double
ramification locus. Double ramification loci and their cycle classes have been studied extensively; see, for instance, \cite{hain-moduli,Panda-pixton,Holmes-Kass-Pagani,BHPSS,chiodo-holmes}. Following the terminology of \cite{Panda-pixton,chiodo-holmes}, we define
the double ramification locus associated with the normal function
\(\nu\) by
\[
    \operatorname{DRL}_Y(\nu)
    :=
    \left\{
        y\in Y
        \ \middle|\
        \nu(y)\simeq\mathcal O_{C_y}
    \right\}.
\]
Equivalently, \(\operatorname{DRL}_Y(\nu)\) is the zero locus of the
section $\nu$.
Assume that \(d\neq0\) and that $a_i:=\frac{(2g-2)m_i}{d}$ for $i=1,\ldots,n$ are positive integers. We denote by $\mathcal P_Y(a_1,\ldots,a_n) := \mu^{-1}\bigl(\mathcal P(a_1,\ldots,a_n)\bigr) \subset Y$ the inverse image of the corresponding Abelian stratum. Since \(\mu\) is étale, \(\mathcal P_Y(a_1,\ldots,a_n)\) is smooth of codimension \(g-1\) (see \cite{Kontsevich}). 
Up to replacing \(\mathbf m\) by \(-\mathbf m\), we may assume \(d>0\). On \(\mathcal M_{g,n}\),the smoothness  and dimension of the zero locus of $\nu$ on the locus smooth curves are known from \cite[Proposition~1.2]{Schmitt}; see also \cite{Holmes–Schmitt} for the relation with the pluricanonical double ramification locus. In the present setting, the point is that this local geometry follows directly from the pointwise rank computation of Theorem~\ref{thm:max-rank-nu}, which also determines the rank away from the double ramification locus.

\begin{corollary}\label{cor:DRL}
Assume that \(d\neq0\) and that $a_i=\frac{(2g-2)m_i}{d}$ fo $i=1,\ldots,n,$
are positive integers. Then the double ramification locus is smooth
of codimension \(g\) at every point
\[
    y\in
    \operatorname{DRL}_Y(\nu)
    \setminus
    \mathcal P_Y(a_1,\ldots,a_n).
\]
Moreover, for every $y\in\mathcal P_Y(a_1,\ldots,a_n),$
there exists a neighbourhood \(V\) of \(y\) such that
\[
    \operatorname{DRL}_Y(\nu)\cap V
    =
    \mathcal P_Y(a_1,\ldots,a_n)\cap V.
\]
In particular, the double ramification locus is smooth of codimension
\(g-1\) along the Abelian stratum.
\end{corollary}

\begin{proof}
Let \(y\in\operatorname{DRL}_Y(\nu)\). By the
description of the zero locus given in Section~2, we have $T_y\operatorname{DRL}_Y(\nu)=\mathcal N_y^\nu.$
Suppose first that $y\notin\mathcal P_Y(a_1,\ldots,a_n).$
By Theorem~\ref{thm:max-rank-nu}, $\operatorname{codim}_{\mathbb C}\mathcal N_y^\nu=g.$
Hence the tangent space of
\(\operatorname{DRL}_Y(\nu)\) at \(y\) has codimension \(g\).
Since \(\operatorname{DRL}_Y(\nu)\) is locally defined as the zero
locus of a holomorphic section of a \(g\)-dimensional family of
complex tori, it has local codimension at most \(g\). Therefore
\(\operatorname{DRL}_Y(\nu)\) is smooth of codimension \(g\) at \(y\).

Assume now that $y\in\mathcal P_Y(a_1,\ldots,a_n).$
Since \(\nu\) vanishes identically along the Abelian stratum,
\[
    \mathcal P_Y(a_1,\ldots,a_n)
    \subseteq
    \operatorname{DRL}_Y(\nu),
\]
and both loci are smooth at \(y\).
Moreover, Theorem~\ref{thm:max-rank-nu} gives $\operatorname{codim}_{\mathbb C}\mathcal N_y^\nu=g-1.$
Thus
\[
    \dim_{\mathbb C}
    T_y\operatorname{DRL}_Y(\nu)
    =
    \dim_{\mathbb C}
    \mathcal P_Y(a_1,\ldots,a_n).
\]
Since $T_y\mathcal P_Y(a_1,\ldots,a_n)
    \subseteq
    T_y\operatorname{DRL}_Y(\nu),$
the two tangent spaces coincide. Hence the inclusion $\mathcal P_Y(a_1,\ldots,a_n) \hookrightarrow \operatorname{DRL}_Y(\nu)$ has invertible differential at \(y\). By the holomorphic inverse function theorem, the inclusion is a local biholomorphism at \(y\). Therefore there exists a neighbourhood \(V\subset Y\) of \(y\) such that 
\[ \operatorname{DRL}_Y(\nu)\cap V = \mathcal P_Y(a_1,\ldots,a_n)\cap V. \]
\end{proof}

In this sense, the
main theorem provides an infinitesimal explanation of the expected and
excess behaviour of the associated double ramification locus.

The paper is organized as follows. In Section~\ref{sec normal fun} we recall the basic theory of normal functions and we introduce the infinitesimal invariant, we define the rank
and the null space \(\mathcal N_y^\nu\). We then discuss, in Section~\ref{sec remarks on deformations}, infinitesimal deformations of
pointed curves and the role of Schiffer deformations. 
In
Section~\ref{section griff formula} we review the Griffiths formula for the
infinitesimal invariant applied on simple tensors.
In Section~\ref{sec abelian strata} we recall
the basic geometry of projectivized strata of abelian differentials and we
explain why the abelian stratum gives rise to a codimension drop of $\mathcal N_y^\nu$.   Finally in Section ~\ref{sec main theorem}, we construct suitable bases of holomorphic differentials
and use the Griffiths formula to prove the main theorem.

\section{Normal functions}\label{sec normal fun}

Let $\pi: \cC\to Y$ be a smooth family of complete complex curves of genus $g \geq 2$, where $Y$ is a connected complex variety, and set $C_y := \pi^{-1}(y)$ the fiber over $y \in Y$.

\subsection{Definition of normal functions}

Let
$$
 \frac{\mathcal{H}}{\mathcal{F} + R^1\pi_* \mathbb{Z}} \longrightarrow Y,
$$
be the jacobian fibration associated to the family $\pi$,
where $\mathcal{H} := R^1\pi_* \mathbb{C} \otimes_\mathbb C \mathcal{O}_Y$ is the flat holomorphic vector bundle associated to the local system $R^1\pi_* \mathbb{C}$, whose fiber at $y \in Y$ is 
$$
\mathcal{H}_y = H^1(C_y, \mathbb{C}),
$$ 
and $\mathcal{F}$ is the is the holomorphic subbundle whose fiber is 
$$
\mathcal{F}_y = H^0(C_y, \omega_{C_y}).
$$
Recall that the Picard bundle at degree $0$ is isomorphic with to the jacobian fibration, that is
\[
    \operatorname{Pic}^0(\pi) \simeq \mathcal{J}
\]

Sections of $\mathcal{J}$ are classically known as \emph{normal functions} (see \cite[Chap.~2, Sec.~7]{Voisin} and \cite[Chap.~9]{Carlson} for a more general definition).

Given $\mathcal{D} \subset \mathcal{C}$ be a relative divisor such that for every $y \in Y$ the divisor $D_y$ of $C_y$ has degree zero, we can define the associated normal function
\[
    \nu : Y \rightarrow \frac{\mathcal{H}}{\mathcal{F} + R^1\pi_* \mathbb{Z}}
\]
setting $\nu(y) = \mathcal{O}_{C_y}(D_y) \in \operatorname{Pic}^0(C_y)$, the Abel-Jacobi map.

\subsection{The Griffiths infinitesimal invariant}\label{sec inf invariant}
In this subsection, we review and elaborate on work of Griffiths\cite{infinitesimal}, Green\cite{Green} and Voisin\cite{Voisin-une-remarque} on invariants of normal functions associated to a family of curves. 

Recall that $\mathcal{H}$ is a flat vector bundle and hence it is equipped with a flat connection $\nabla$
\[
\nabla : \mathcal{H} \rightarrow \mathcal{H} \otimes \Omega^1_Y
\]
called the Gauss-Manin connection.

Set $\mathcal{H}^{0,1} := \frac{\mathcal{H}}{\mathcal{F}} \simeq \mathcal{F}^*$. Consider the bundle map, representing the IVHS map,
\[
    \overline{\nabla} : \mathcal{F} \to \mathcal{H}^{0,1} \otimes \Omega_Y^1,
\]
defined, over a point $y \in Y$, by
\[
    \overline\nabla_\xi(\omega) := \kappa_y(\xi) \cdot \omega \in H^1(\mathcal{O}_{C_y}) \simeq H^{0,1}(C)
\]
where
\[
    \kappa_y: T_yY \rightarrow H^1(T_{C_y})
\]
is the Kodaira-Spencer map and $\cdot$ is the cup product.
Following Voisin \cite[Chap.~2, Sec.~7]{Voisin} and Green~\cite{Green}, given a normal function $\nu$, we define the infinitesimal invariant. Let $\tilde{\nu}$ a local lifting of $\nu$ in $\mathcal{H}$, and consider the projection $\overline{\nabla \tilde{\nu}}$ of $\nabla \tilde{\nu}$ in the quotient
\[
    \frac{\mathcal{H}}{\mathcal{F}} \otimes \Omega^1_Y \simeq \mathcal{H}^{0,1} \otimes \Omega^1_Y.
\]
Then consider the class
\[
    [\overline{\nabla \tilde{\nu}}] \in \frac{\mathcal{H}^{0,1}\otimes \Omega_Y}{\overline\nabla \mathcal{F}}.
\]
This is well defined and does not depend on the choice of lifting (see Lemma~7.7 in \cite{Voisin}). We denote by $\delta\nu$ this class and call it the infinitesimal invariant. It is useful for the computation to define also the dual version of $\delta\nu$. Let us consider the transpose of the Gauss–Manin connection
\[
    \overline{\nabla}^t : \mathcal{F} \otimes T_Y\to \mathcal{H}^{0,1}
\] 
Choose a point $y \in Y$ , then $\overline{\nabla}^t : H^0(C_y, \omega_{C_y}) \otimes T_yY \to H^1(C_y, \mathcal{O}_{C_y})$ is given by
\begin{equation}\label{Gauss_Manin}
    \overline{\nabla}^t \left( \sum_i \omega_i \otimes \xi_i \right) = \sum_i \overline{\nabla}_{\xi_i}(\omega_i) = \sum_i \kappa_y(\xi_i) \cdot \omega_i.
\end{equation}
By duality we have
\[
    \frac{\mathcal{H}^{0,1} \otimes \Omega_Y}{\overline\nabla \mathcal{F}} \simeq \left(\ker(\overline{\nabla}^t)\right)^*,
\]
then we can consider $\delta\nu$ as a section of this last vector bundle. Over a point $y \in Y$, we have:
\begin{equation}\label{abel-integrals}
    \delta\nu(y)\left( \sum_i \omega_i \otimes \xi_i \right) = \sum_i \int_C \nabla_{\xi_i} \tilde{\nu}\wedge \omega_i,
\end{equation}
	where $\sum_i \kappa(\xi_i) \cdot \omega_i = 0$.

\subsection{Rank of a normal function}

Let $\mathcal{J}$ be the jacobian fibration. We have a diffeomorphism 
\[
    \mathcal{J} \stackrel{f}\rightarrow \frac{R^1\pi_*\mathbb{R}}{R^1\pi_*\mathbb{Z}}.
\]

Since $R^1\pi_*\mathbb{R}$ is a flat bundle, its locally constant
sections induce a natural foliation \(\mathscr L\) of \(\mathcal J\).
Following Hain \cite[Sec.~2.5]{Hain}, the leaves of \(\mathscr L\) are
complex submanifolds of \(\mathcal J\), although the foliation itself
need not be holomorphic.
Let
\[
    p:\mathcal J\longrightarrow Y
\]
be the projection. For \(v\in\mathcal J_y\), let
\(\mathscr L_v\subset T_v\mathcal J\) denote the tangent space to the
locally constant leaf through \(v\). The differential of \(p\) induces
a complex-linear isomorphism
\[
    dp_v:\mathscr L_v\xrightarrow{\sim}T_yY.
\]
Hence the exact sequence
\[
    0
    \longrightarrow
    T_v\mathcal J_y
    \longrightarrow
    T_v\mathcal J
    \xrightarrow{\,dp_v\,}
    T_yY
    \longrightarrow
    0
\]
is split by \(\mathscr L_v\). Using the natural identification $T_v\mathcal J_y
    \simeq
    \mathcal H_y/\mathcal F_y,$
we denote by
\[
    \operatorname{pr}_v:
    T_v\mathcal J
    \longrightarrow
    \mathcal H_y/\mathcal F_y
\]
the corresponding projection. Its kernel is precisely
\(\mathscr L_v\), and \(\operatorname{pr}_v\) is complex linear.

\begin{definition}\label{def:rank}
Let \(\nu:Y\to\mathcal J\) be a normal function. The rank of \(\nu\)
at \(y\in Y\) is
\[
    \operatorname{rk}_y(\nu)
    :=
    \operatorname{rk}_{\mathbb C}
    \left(
        \operatorname{pr}_{\nu(y)}\circ \, d_y\nu
    \right).
\]
\end{definition}

\begin{definition}\label{def:infinitesimal-local-constancy-space}
We define the \emph{null space} of \(\nu\) at \(y\) by
\[
    \mathcal N_y^\nu
    :=
    \ker
    \left(
        \operatorname{pr}_{\nu(y)}\circ \, d_y\nu
    \right)
    \subseteq T_yY.
\]
\end{definition}

Since $\operatorname{pr}_{\nu(y)}\circ \,d_y\nu:
    T_yY\longrightarrow\mathcal H_y/\mathcal F_y$
is complex linear, \(\mathcal N_y^\nu\) is a complex vector subspace of
\(T_yY\), and
\begin{equation}\label{codim < g}
    \operatorname{rk}_y(\nu)
    =
    \operatorname{codim}_{\mathbb C}\mathcal N_y^\nu
    \leq g.
\end{equation}

\begin{remark}\label{rem:real-rank} Let \(U\subset Y\) be simply connected and choose a flat trivialization, as a smooth group bundle, 
\[
f_U:\mathcal J|_U \xrightarrow{\sim} T^{2g}\times U.
\] Set $\phi_U:=\operatorname{pr}_1\circ f_U: \mathcal J|_U\longrightarrow T^{2g}.$ 
The fibres of \(\phi_U\) are precisely the locally constant leaves. Consequently, $\ker_{\mathbb R}d_y(\phi_U\circ\nu) = \bigl(\mathcal N_y^\nu\bigr)_{\mathbb R},$
and hence
\[
\operatorname{rk}_{\mathbb R} d_y(\phi_U\circ\nu) = 2\,\operatorname{rk}_y(\nu). 
\] 
Thus $\operatorname{rk}_y(\nu) = \frac12 \operatorname{rk}_{\mathbb R} d_y(\phi_U\circ\nu).$
\end{remark}


\subsection{Locally constant normal functions}

\begin{definition}\label{def:locally-constant}
A normal function \(\nu\) is said to be \emph{locally constant near}
\(y\in Y\) if, after restricting to a sufficiently small neighbourhood
\(U\) of \(y\), its image \(\nu(U)\) is contained in a single leaf of
the foliation \(\mathscr L\).
\end{definition}

An important example of a locally constant normal function is given by
a torsion section.

\begin{remark}\label{rem:global-rank}
Following Hain \cite[Sec.~3]{Hain}, the rank of a normal function is
defined by
\[
    \operatorname{rk}(\nu)
    :=
    \max_{y\in Y}\operatorname{rk}_y(\nu).
\]
This is an intrinsic notion and is independent of the choice of a
local flat trivialization. In particular, $\operatorname{rk}(\nu)=0$ if and only if \(\nu\) is locally constant.
\end{remark}

\begin{definition}\label{def:locally-constant-direction}
Let \(y\in Y\) and \(\xi\in T_yY\). We say that \(\nu\) is
\emph{locally constant to first order along \(\xi\)} if $d_y\nu(\xi)\in\mathscr L_{\nu(y)}.$
\end{definition}

By construction, $d_y\nu(\xi)\in\mathscr L_{\nu(y)}$ if and only if $\operatorname{pr}_{\nu(y)}
    \bigl(d_y\nu(\xi)\bigr)=0.$
Therefore we have,
\begin{equation}\label{eq:xi in N}
    \xi\in\mathcal N_y^\nu
    \quad\Longleftrightarrow\quad
    \nu\text{ is locally constant to first order along }\xi.
\end{equation}

\begin{remark}\label{rem:loc-constant-delta}
If \(\nu\) is locally constant to first order along
\(\xi\in T_yY\), then its Griffiths infinitesimal invariant vanishes in
that direction. More precisely, for every
\(\omega\in H^0(C_y,\omega_{C_y})\) such that $\kappa_y(\xi)\cdot\omega=0,$
we have
\[
    \delta\nu(y)(\omega\otimes\xi)=0.
\]
In particular, if \(\nu\) is locally constant, then $\delta\nu=0.$
\end{remark}

\section{Remarks on deformations of curves}\label{sec remarks on deformations}

Since we are interested in normal functions defined on families of \emph{marked curves}, it is important to clarify how the Griffiths infinitesimal invariant interacts with deformations of the marked points.

Let \(C\) be a smooth curve of genus \(g\geq 2\), with distinct marked
points \(p_1,\ldots,p_n\), and set $E:=p_1+\cdots+p_n.$ We denote by \(\mathcal{M}_g\) and \(\mathcal{M}_{g,n}\), respectively, the complex-analytic
Deligne--Mumford stacks of smooth curves and smooth \(n\)-pointed curves
of genus \(g\). By abuse of notation, we use the same symbols for the
algebraic moduli stacks and their analytifications.

\subsection{Deformations of \(\mathcal{M}_{g,n}\)}

First-order deformations of the pointed curve
\((C;p_1,\ldots,p_n)\) are parametrized by
\[
    T_{[C;p_1,\ldots,p_n]}\mathcal M_{g,n}
    \simeq H^1(C,T_C(-E)).
\]
Let $q:\mathcal M_{g,n}\longrightarrow\mathcal M_g$
be the forgetful morphism. Its differential at
\([C;p_1,\ldots,p_n]\) is identified with the natural map
\begin{equation}\label{varphi}
    dq:
    H^1(C,T_C(-E))
    \longrightarrow
    H^1(C,T_C)
\end{equation}
induced by the inclusion $T_C(-E)\hookrightarrow T_C.$

The exact sequence
\[
    0\longrightarrow T_C(-E)
    \longrightarrow T_C
    \longrightarrow T_C|_E
    \longrightarrow0
\]
and the vanishing \(H^0(C,T_C)=0\), valid since \(g\geq2\), give
\begin{equation}\label{Exact_seq_def}
    0\longrightarrow
    \bigoplus_{i=1}^n T_{p_i}C
    \longrightarrow
    H^1(C,T_C(-E))
    \xrightarrow{\,dq\,}
    H^1(C,T_C)
    \longrightarrow0.
\end{equation}
Thus \(\ker(dq)\) consists precisely of the infinitesimal deformations
for which the complex structure of \(C\) remains fixed and only the
marked points move.

Since the modular map \(\mu\) is étale, its
differential identifies $T_yY
    \simeq
    H^1(C,T_C(-E)).$
Let
\[
    \kappa_y:T_yY\longrightarrow H^1(C,T_C)
\]
be the Kodaira--Spencer map of the underlying family of curves
\(\pi:\mathcal C\to Y\), obtained by forgetting the marked sections.
Since the classifying map of the underlying family is \(q\circ\mu\),
we have
\[
    \kappa_y=dq\circ d\mu_y.
\]
Hence, under the above identification
\(T_yY\simeq H^1(C,T_C(-E))\), we shall simply write
\begin{equation}\label{KS-forgetful}
    \kappa_y(\xi)=dq(\xi).
\end{equation}

In particular, the infinitesimal variation of the Hodge structure of
\(C\) in the direction
\(\xi\in T_yY\simeq H^1(C,T_C(-E))\) depends only on the deformation of
the underlying curve, namely on \(dq(\xi)\). Thus the cup product
appearing in \eqref{Gauss_Manin} is
\[
    \kappa_y(\xi)\cdot\omega
    =
    dq(\xi)\cdot\omega
    \in H^1(C,\mathcal O_C).
\]

\subsection{The infinitesimal invariant on marked deformations}

We now record the contribution of the marked points to the infinitesimal
invariant. By the exact sequence \eqref{Exact_seq_def}, a vector
\[
    v=(v_1,\ldots,v_n)\in\bigoplus_{i=1}^nT_{p_i}C
\]
represents an infinitesimal deformation in which the complex structure
of \(C\) remains fixed and only the marked points move. In particular,
for every \(\omega\in H^0(C,\omega_C)\),
\[
    dq(v)\cdot\omega=0.
\]
Then the Griffiths infinitesimal can be evaluated on  \(v \otimes\omega\). More precisely, let $\mathcal D=\sum_{i=1}^n a_i p_i$ with $\sum_{i=1}^n a_i=0$
be a relative divisor of degree zero, and let \(\nu_{\mathcal D}\) be
the associated normal function. A direct local computation of the
corresponding integral \eqref{abel-integrals} gives
\begin{equation}\label{formula-vertical-inf-invariant}
    \delta\nu_{\mathcal D}(y)(\omega\otimes v)
    =
    \sum_{i=1}^n
    a_i\,\omega(p_i)(v_i).
\end{equation}
Here \(\omega(p_i)(v_i)\) denotes the natural pairing between
\(\omega(p_i)\in T_{p_i}^*C\) and \(v_i\in T_{p_i}C\). Equivalently,
if \(z\) is a local coordinate centered at \(p_i\), setting $\omega=f_i(z)\,dz$ and $v_i=\dot z_i\frac{\partial}{\partial z}$,
then
\[
    \omega(p_i)(v_i)=f_i(0)\dot z_i.
\]

In particular, for the normal function
\[
    \nu=(2g-2)\psi_{\mathbf m}-d\kappa
\]
considered in this paper, the canonical section \(\kappa\) is constant
along deformations in \(\ker(dq)\). Hence
\begin{equation}\label{formula-vertical-main-normal-function}
    \delta\nu(y)(\omega\otimes v)
    =
    (2g-2)
    \sum_{i=1}^n
    m_i\,\omega(p_i)(v_i).
\end{equation}

\subsection{Schiffer Deformations}\label{subsec schiffer def}

Let \(C\) be a smooth complex curve of genus \(g\geq 2\), and let
\(\omega\in H^0(C,K_C)\) be a non-zero holomorphic \(1\)-form.
Multiplication by \(\omega\) gives an exact sequence
\[
    0\longrightarrow T_C
    \xrightarrow{\cdot\omega}
    \mathcal O_C
    \longrightarrow
    \mathcal O_Z
    \longrightarrow0,
\]
where \(Z\) denotes the zero scheme of \(\omega\). Since
\(H^0(C,T_C)=0\), the associated long exact sequence gives
\begin{equation}\label{cohomo-seq-cup-prod-omega}
    0\longrightarrow H^0(C,\mathcal O_C)
    \longrightarrow H^0(C,\mathcal O_Z)
    \xrightarrow{\partial_Z} H^1(C,T_C)
    \xrightarrow{\cdot\omega} H^1(C,\mathcal O_C)
    \longrightarrow0.
\end{equation}
In particular,
\[
    \operatorname{Im}(\partial_Z)
    =
    H^1(C,T_C)_\omega
    :=
    \ker\!\left(
        H^1(C,T_C)
        \xrightarrow{\cdot\omega}
        H^1(C,\mathcal O_C)
    \right),
\]
and hence
\[
    \dim H^1(C,T_C)_\omega=2g-3.
\]

Following \cite{Col-Fre-Pir}, we describe the connecting morphism
\(\partial_Z\) in the Dolbeault model. Let
\(\theta\in H^0(\mathcal O_Z)\), choose a smooth extension
\(\widetilde\theta\) of \(\theta\) to a neighbourhood of \(Z\), and let
\(\rho\) be a smooth bump function which is identically \(1\) near \(Z\).
Writing locally $\omega=f(z)\,dz,$
the expression
\[
    \frac{1}{\omega}
    :=
    \frac{1}{f(z)}\frac{\partial}{\partial z}
\]
defines a meromorphic vector field with poles along \(Z\). The class
\(\partial_Z(\theta)\) is represented in Dolbeault cohomology by
\[
    \partial_Z(\theta)
    =
    \left[
        \frac{\bar\partial(\rho\widetilde\theta)}{\omega}
    \right]_{\mathrm{Dolb}}.
\]

\begin{definition} 
If $\theta\in H^0(\cO_Z)$ is such that $\theta(q)= 1$ at a point $q\in Z$ and $0$ on the other points and \(\rho_q\) is a bump function supported in a small neighbourhood of \(q\) and identically
\(1\) near \(q\), we call
\begin{equation} \label{schiffer}
\eta_q= \left[\frac{\debar \rho}{\omega}\right]_\mathrm{Dolb} \in H^1(C,T_C)
\end{equation}
a Schiffer deformation to $q$.
\end{definition}

More generally, let \(q\in Z\) and set $m_q:=\operatorname{ord}_q(\omega).$
In a local coordinate \(z\) centered at \(q\), a section of
\(\mathcal O_Z\) supported at \(q\) is represented by
\[
    \theta(z)
    =
    a_0+a_1z+\cdots+a_{m_q-1}z^{m_q-1}.
\]
Therefore its image under \(\partial_Z\) is represented by
\begin{equation}\label{def_on_ker}
    \left[
        \frac{
            \bar\partial\rho_q\,
            (a_0+a_1z+\cdots+a_{m_q-1}z^{m_q-1})
        }{\omega}
    \right]_{\mathrm{Dolb}}.
\end{equation}
By exactness of \eqref{cohomo-seq-cup-prod-omega}, the image of the
constant sections
\[
    H^0(C,\mathcal O_C)\simeq\mathbb C
    \longrightarrow H^0(C,\mathcal O_Z)
\]
gives the trivial class in \(H^1(C,T_C)\).

\begin{remark}\label{rem:schiffer-pointed}
Let \(p_i\) be a marked for some
\(i\). Although the Schiffer deformation \(\eta_{p_i}\) is centered at
\(p_i\), its Dolbeault representative $\frac{\bar\partial\rho_{p_i}}{\omega}$
vanishes in a neighbourhood of \(p_i\), since \(\rho_{p_i}\) is identically
\(1\) there. Choosing the support of \(\rho_{p_i}\) sufficiently small, the
representative vanishes in a neighbourhood of \(E\) and may therefore
be regarded as a \(T_C(-E)\)-valued form. Hence it defines a class
\[
    \eta_{p_i}\in H^1(C,T_C(-E))
\]
whose image under the map \eqref{varphi}
is the usual Schiffer deformation \eqref{schiffer}. We shall use the
same notation \(\eta_{p_i}\) for both classes.
\end{remark}

\begin{remark}\label{rem:vertical-schiffer}
Let $p_i$ a marked point for some $i$. If we consider a tangent vector
\[
    v_i\in T_{p_i}C
    \subset H^1(C,T_C(-E))
\]
Its Dolbeault representative can be described explicitly. Choose a
local coordinate \(z\) centered at \(p_i\), write
\[
    v_i=a_i\frac{\partial}{\partial z}\Big|_{p_i},
\]
and let \(\rho_i\) be a smooth bump function supported in a sufficiently
small neighbourhood of \(p_i\) and identically equal to \(1\) near
\(p_i\). If
\[
    \widetilde v_i
    =
    a_i\frac{\partial}{\partial z}
\]
denotes the corresponding local extension, then the image of \(v_i\)
under the connecting homomorphism is represented by
\[
    \left[
        \bar\partial(\rho_i\widetilde v_i)
    \right]_{\mathrm{Dolb}}
    =
    \left[
        a_i\,\bar\partial\rho_i
        \frac{\partial}{\partial z}
    \right]_{\mathrm{Dolb}}
    \in H^1(C,T_C(-E)).
\]

It is nevertheless non-trivial in \(H^1(C,T_C(-E))\).

Thus these classes represent infinitesimal variation of the marked points
with the complex structure of \(C\) fixed. 

By a slight abuse of terminology, we shall refer to these localized
marked deformations as \emph{vertical Schiffer deformations}, in order
to distinguish them from the usual Schiffer deformations in
\(H^1(C,T_C)\).
\end{remark}

\section{Griffiths formula}\label{section griff formula}

Using the previous notations of Subsection \ref{subsec schiffer def} we give the Griffiths formula  for a decomposable tensors when the normal function is {\em supported on the zeroes of $\omega$} (see equation~(6.18) in \cite{infinitesimal}).

Let $\pi:\mathcal C\longrightarrow Y$ be a smooth family of complex projective curves of genus \(g\), and fix
\(y\in Y\). Set $C .= C_y$

\begin{theorem}\label{Griff_theorem}
Let \(\nu\) be a normal function which, locally near \(y\),
is represented by a relative divisor $\mathcal D=\sum_{i=1}^k a_i p_i$ with $\sum_{i=1}^k a_i=0,$
where the \(p_i:Y\to\mathcal C\) are local sections of \(\pi\). Thus,
\[
    \nu(y)=\mathcal O_C(D) \in \operatorname{Pic}^0(C),
    \qquad
    D=\sum_{i=1}^k a_i p_i(y).
\]

Let $\xi\in T_yY$ and $\omega\in H^0(C,\omega_C)$
such that $\kappa_y(\xi)\cdot\omega=0$ in $H^1(C,\mathcal O_C).$
Choose a smooth function \(h\in C^\infty(C)\) such that, in the Dolbeault model,
\[
    \bar\partial h=\kappa_y(\xi)\cdot\omega.
\]
Then
\begin{equation}\label{Griffiths formula completa}
    \delta\nu(y)(\omega\otimes\xi)
    =
    \sum_{i=1}^k
    a_i\,\omega(p_i)\bigl(p_i'(\xi)\bigr)
    +
    \sum_{i=1}^k a_i h(p_i),
\end{equation}
where
\[
    p_i'(\xi)\in T_{p_i}C
\]
denotes the infinitesimal variation of the section \(p_i\) in the
direction \(\xi\).

Let \(Z\) be the zero scheme of \(\omega\). If $\operatorname{Supp}(D)\subseteq Z,$
then \(\omega(p_i)=0\) for every \(i\), and hence
\begin{equation}\label{Griff_formula}
    \delta\nu(y)(\omega\otimes\xi)
    =
    \sum_{i=1}^k a_i h(p_i).
\end{equation}
\end{theorem}



\begin{remark} The  condition that the support of $D$ is contained in the set of zeros of $\omega$ allows to eliminate the  derivative that appears from the deformation of the divisor.
\end{remark}

For later purposes we want also remark the following.

\begin{remark}
The quantity $\sum_{i=1}^k a_i h(p_i)$
is independent of the choice of the solution \(h\) of $\bar\partial h=\kappa_y(\xi)\cdot\omega,$

since two such solutions differ by a constant and
\(\sum_i a_i=0\).
If, moreover, $\operatorname{Supp}(D)\subseteq  Z,$
then this quantity coincides with
\(\delta\nu(y)(\omega\otimes\xi)\). Consequently, it is also
independent of the choice of a linearly equivalent degree-zero
representative of \(\nu(y)\).
\end{remark}

\section{Strata of Abelian differentials}\label{sec abelian strata}
In this section we recall the basic geometry of strata of Abelian
differentials and explain their relation with the normal function
\(\nu\) introduced in \eqref{eq:normal-function}.

\subsection{Projectivized strata of Abelian differentials}

Following the approach introduced by Kontsevich~\cite{Kontsevich}, we describe the stratification of the moduli space of holomorphic $1$-forms according to the multiplicities of their zeros. 

For any integer $g \geq 2$, define the space $\mathcal{H}_g$ as the moduli stack of pairs $(C,\omega)$, where $C$ is a smooth compact complex curve of genus $g$, and $\omega$ is a holomorphic $1$-form on $C$ (an Abelian differential) which is not identically zero. Equivalently, $\mathcal{H}_g$ can be described as the total space of the Hodge bundle over $\mathcal{M}_g$ with the zero section removed.


The space $\mathcal{H}_g$ is therefore a smooth Deligne--Mumford stack of complex dimension
\[
\dim_{\mathbb{C}} \mathcal{H}_g = 4g-3.
\]
There is a natural projection
\begin{equation}\label{proj H_g to M_g}
\mathcal{H}_g \longrightarrow \mathcal{M}_g, \qquad (C,\omega) \longmapsto [C],
\end{equation}
whose fibre over a point $[C]\in\mathcal{M}_g$ is the punctured vector space $H^0(C,\omega_C) \setminus \{0\}$.

\medskip

The stack $\mathcal{H}_g$ admits a natural stratification according to the multiplicities of the zeros of $\omega$. Let $(m_1, \dots, m_n)$ be an ordered \(n\)-tuple of positive integers such that $m_1 \leq \dots \leq m_n$ and satisfying
\[
\sum_{i=1}^{n} m_i = 2g-2.
\]
We define $\mathcal{H}(m_1, \dots, m_n)$ to be the locus in $\mathcal{H}_g$ consisting of pairs $(C,\omega)$ such that $\omega$ has precisely $n$ zeros with multiplicities $m_1, \dots, m_n$. The notation is symmetric with respect to permutations of the indices: for any permutation $\sigma \in S_n$ we have
\[
\mathcal{H}(m_1, \dots, m_n) = \mathcal{H}(m_{\sigma(1)}, \dots, m_{\sigma(n)}).
\]

\medskip

The dimension of these strata is given by the following classical result (see \cite{Polishchuk}, \cite{VM_quadratic_diff}).
\begin{theorem}\label{Dim_stratum}
The stratum $\mathcal{H}(m_1,\dots,m_n)$ is a smooth
Deligne--Mumford stack of complex dimension
\[
\dim_{\mathbb{C}}\mathcal{H}(m_1,\dots,m_n)
=
2g+n-1.
\]
\end{theorem}

\medskip

The natural action of $\mathbb{C}^{*}$ on
$\mathcal{H}(m_1,\dots,m_n)$, given by rescaling the differential,
defines the projectivized stratum
\[
\mathbb{P}\mathcal{H}(m_1,\dots,m_n)
:=
\mathcal{H}(m_1,\dots,m_n)/\mathbb{C}^{*}.
\]
Its points are pairs $(C,[\omega])$, where the zero divisor of $\omega$
has multiplicities $m_1,\dots,m_n$, considered without ordering.
Moreover, we have:
\begin{equation}\label{dimension proj H}
\dim_{\mathbb{C}} \mathbb{P}\mathcal{H}(m_1, \dots, m_n) = \dim_{\mathbb{C}} \mathcal{H}(m_1, \dots, m_n) - 1 = 2g + n - 2.
\end{equation}

Since the marked points in $\mathcal{M}_{g,n}$ are labelled, we replace $\mathbb{P}\mathcal{H}(m_1,\ldots,m_n)$ by the canonical finite étale cover $\mathcal{P}(m_1,\ldots,m_n)$
parametrizing tuples $(C,p_1,\ldots,p_n,[\omega])$ such that the points $p_i$ are pairwise distinct and 
$\operatorname{div}(\omega) = \sum_{i=1}^n m_i p_i$. This replacement is relevant when some of the multiplicities $m_i$ coincide. 
There is a commutative diagram 
\[ 
\begin{array}{ccc} \mathcal{P}(m_1,\ldots,m_n) & \longrightarrow & \mathcal{M}_{g,n} \\[4pt] \downarrow && \downarrow \\[4pt] \mathbb{P}\mathcal{H}(m_1,\ldots,m_n) & \longrightarrow & \mathcal{M}_g, \end{array} 
\] where the vertical morphisms forget, respectively, the labelling of the zeros and the marked points.
The image of the morphism
\[
\mathcal P(m_1,\ldots,m_n)
\longrightarrow
\mathcal M_{g,n},
\qquad
(C,p_1,\ldots,p_n,[\omega])
\longmapsto
(C,p_1,\ldots,p_n).
\]
is the locus of pointed curves satisfying $\omega_C
\simeq
\mathcal O_C\left(\sum_{i=1}^n m_i p_i\right).$
Moreover, the projective differential $[\omega]$ is uniquely determined
by the pointed curve. Indeed, two nonzero holomorphic differentials
having the same zero divisor differ by a nonzero constant. Therefore,
the morphism identifies
$\mathcal P(m_1,\ldots,m_n)$ with its image in
$\mathcal M_{g,n}$ as a substack.

In particular, \[ \operatorname{codim}_{\mathcal{M}_{g,n}} \mathcal{P}(m_1,\ldots,m_n) = g-1. \]



\subsection{The Abelian stratum and the null space}
\label{sec:abelian-rank-zero}

We now relate the strata introduced above to the normal function
\(\nu\) defined in \eqref{eq:normal-function}. Let \(y\in Y\) and set $C:=C_y, p_i:=p_i(y)$ and $D:=D_y=\sum_{i=1}^n m_i p_i.$

Assume \(d\neq 0\) and set
\begin{equation}\label{a_i}
    a_i:=\frac{(2g-2)m_i}{d},
    \qquad i=1,\ldots,n.
\end{equation}
Suppose that \(a_1,\ldots,a_n\) are positive integers and that
\[
    [C;p_1,\ldots,p_n]
    \in
    \mathcal P(a_1,\ldots,a_n).
\]
By definition of the stratum, $\omega_C
    \simeq
    \mathcal O_C\left(\sum_{i=1}^n a_i p_i\right).$
Hence
\[
    \nu(y)
    =
    \mathcal O_C\bigl((2g-2)D\bigr)
    \otimes\omega_C^{-d}
    \simeq
    \mathcal O_C.
\]

Thus \(\nu\) vanishes identically along the inverse image in \(Y\) of
\(\mathcal P(a_1,\ldots,a_n)\). Consequently, its differential vanishes
on the tangent directions to the stratum. Using the étale identification
of tangent spaces given by the modular map, we obtain
\[
    T_{[C;p_1,\ldots,p_n]}
    \mathcal P(a_1,\ldots,a_n)
    \subseteq
    \mathcal N_y^\nu .
\]
Since
\[
    \operatorname{codim}_{\mathcal M_{g,n}}
    \mathcal P(a_1,\ldots,a_n)
    =
    g-1,
\]
we conclude that
\begin{equation}\label{disug g-1}
    \operatorname{codim}_{T_{[C;p_1,\ldots,p_n]}\mathcal M_{g,n}}
    \mathcal N_y^\nu
    \leq g-1.
\end{equation}

Therefore the Abelian stratum provides a geometric locus where the
maximal possible codimension \(g\) of \(\mathcal N_y^\nu\) cannot occur.

\section{Proof of the Main Theorem}\label{sec main theorem}

We now prove Theorem~\ref{thm:max-rank-nu}. The main ingredient is the
construction of suitable bases of \(H^0(C,\omega_C)\), adapted to the
position of the marked divisor. 
The purpose of the construction is to choose an ordered basis $\{\omega_1,\ldots,\omega_g\}$
together with suitable Schiffer deformations
\[
    \eta_1,\ldots,\eta_g
    \in H^1(C,T_C(-E))
\]
for which the Griffiths infinitesimal invariant has a triangular
behaviour. More precisely, the ordering will allow us to test a linear
combination
\[
    \eta=\sum_{i=1}^g\alpha_i\eta_i
\]
successively against
\(\omega_g,\omega_{g-1},\ldots,\omega_1\), and to prove by reverse
induction that all the coefficients \(\alpha_i\) vanish.

Let $\pi\colon (\mathcal{C},p_1,\ldots,p_n)\longrightarrow Y$ as in \eqref{definition of family}. Fix \(y\in Y\) and, for simplicity of notation, set $C:=C_y, p_i = p_i(y)$ and $D:=D_y=\sum_{i=1}^n m_i p_i$ for a fixed $n$-tuple of non-zero integers $\mathbf{m} = (m_1,\dots,m_n) \in \mathbb{Z}^n$ such that $m_1\leq\dots\leq m_n$ with $\sum_i m_i = d$.  By hypothesis,
\[
\operatorname{Supp}(D)=\{p_1,\ldots,p_n\}.
\]
Let $\psi_\mathbf{m} \in \operatorname{Pic}^d(\pi)$ as in \eqref{def-psi_m}. Then, let $\nu$ be the associated normal function \eqref{eq:normal-function}. 

Recall that
\[
    \nu(y)
    =
    \mathcal O_C\bigl((2g-2)D\bigr)
    \otimes\omega_C^{-d}.
\]

Since the rank of a normal function is invariant under multiplication
by a non-zero integer, we may assume that
\[
    \gcd(m_1,\ldots,m_n)=1.
\]
Moreover, replacing \(\mathbf m\) by \(-\mathbf m\), if necessary, replaces \(\nu\) by \(-\nu\) and therefore leaves \(\mathcal N_y^\nu\) unchanged. Hence we may assume $d\geq0.$
As before, if $d\neq 0$ set $a_i:=\frac{(2g-2)m_i}{d}$, for $i=1,\ldots,n.$
Set $E:=p_1+\cdots+p_n$
and let
\[
    r:=
    \operatorname{codim}_{H^0(C,\omega_C)}
    H^0(C,\omega_C(-E)).
\]
Thus
\[
    \dim H^0(C,\omega_C(-E))=g-r.
\]

\subsection{Construction of the basis of $H^0(C,\omega_C)$}\label{costruction basis}


Consider the evaluation map
\[
    \operatorname{ev}_E:
    H^0(C,\omega_C)
    \longrightarrow
    \bigoplus_{i=1}^n \omega_C|_{p_i},
    \qquad
    \omega\longmapsto
    \bigl(\omega(p_1),\ldots,\omega(p_n)\bigr).
\]
By definition, its rank is \(r\). After relabelling the marked points
if necessary, we may assume that evaluation at
\(p_1,\ldots,p_r\) gives \(r\) independent conditions. We can therefore
choose linearly independent forms
\[
    \omega_1,\ldots,\omega_r
    \in H^0(C,\omega_C)
\]
such that
\begin{equation}\label{eq:triangular-evaluation}
    \omega_i(p_j)=0
    \quad\text{for }j<i,
    \qquad
    \omega_i(p_i)\neq0,
    \quad i=1,\ldots,r.
\end{equation}
If
\[
    H^0(C,\omega_C(-E))=0,
\]
then \(r=g\), and the differentials $\omega_1,\ldots,\omega_g$
already form the required adapted basis of \(H^0(C,\omega_C)\).
No further choice is necessary. Notice that in this case necessarily \(n\geq g\), since the evaluation
map at the marked points has rank \(g\).
Assume now that
\[
    H^0(C,\omega_C(-E))\neq0,
\]
equivalently \(r<g\). It remains to choose $\omega_{r+1},\ldots,\omega_g$
forming a basis of \(H^0(C,\omega_C(-E))\). 
The choice of these remaining differentials will be given by the
coefficients appearing in the Griffiths formula. More precisely, for
a Schiffer deformation centered at \(p_j\), we shall need to ensure
that
\[
    (2g-2)m_j
    -
    d\,\operatorname{ord}_{p_j}(\omega_i)
    \neq0.
\]
Accordingly, we distinguish the following three cases.

\textbf{Case 1: \(D\) is not effective.}
Choose an index \(s\) such that \(m_s<0\). We choose a basis
\[
    \omega_{r+1},\ldots,\omega_g
\]
of \(H^0(C,\omega_C(-E))\), ordered by increasing vanishing order at
\(p_s\):
\[
    \operatorname{ord}_{p_s}(\omega_{r+1})
    <
    \cdots
    <
    \operatorname{ord}_{p_s}(\omega_g).
\]
Since \(d\geq0\), \(m_s<0\), and
\(\operatorname{ord}_{p_s}(\omega_i)>0\), one has
\[
    (2g-2)m_s
    -
    d\,\operatorname{ord}_{p_s}(\omega_i)
    \neq0.
\]
for every \(i=r+1,\ldots,g\).

\medskip
\noindent
\textbf{Case 2: \(D\) is effective and \(a_s\notin\mathbb Z\) for some \(s\).} Observe that \(d>0\), so the numbers
\(a_i\) are defined.
Choose such an index \(s\) and let
\[
    \omega_{r+1},\ldots,\omega_g
\]
be a basis of \(H^0(C,\omega_C(-E))\), ordered by increasing vanishing
order at \(p_s\).

For every \(i=r+1,\ldots,g\),
\[
    (2g-2)m_s
    -
    d\,\operatorname{ord}_{p_s}(\omega_i)
    \neq0,
\]
since equality would imply
\[
    a_s
    =
    \frac{(2g-2)m_s}{d}
    =
    \operatorname{ord}_{p_s}(\omega_i)
    \in\mathbb Z,
\]
contrary to the choice of \(s\).

\medskip
\noindent
\textbf{Case 3: \(D\) is effective and \(a_i\in\mathbb Z\) for every \(i\).}
Set
\[
    H_0:=H^0(C,\omega_C(-p_1-\cdots-p_n)).
\]
For \(j=1,\ldots,n-1\), define
\[
    H_j
    :=
    H^0\left(
        C,
        \omega_C
        \left(
            -\sum_{i=1}^j a_i p_i
            -\sum_{i=j+1}^n p_i
        \right)
    \right),
\]
and finally set
\[
    H_n
    :=
    H^0\left(
        C,
        \omega_C
        \left(
            -\sum_{i=1}^n a_i p_i
        \right)
    \right).
\]
We obtain a filtration $H_0\supseteq H_1\supseteq\cdots\supseteq H_n.$
For each \(j=1,\ldots,n\), set
\[
    Q_j:=H_{j-1}/H_j,
    \qquad
    d_j:=\dim Q_j.
\]
Choose a basis $\overline{\omega}_{j,1},
    \ldots,
    \overline{\omega}_{j,d_j}$
of \(Q_j\), adapted to the vanishing order at \(p_j\), and choose
lifts
\[
    \omega_{j,1},\ldots,\omega_{j,d_j}
    \in H_{j-1}.
\]
We may arrange them so that
\[
    \operatorname{ord}_{p_j}(\omega_{j,1})
    <
    \cdots
    <
    \operatorname{ord}_{p_j}(\omega_{j,d_j}).
\]
We order the pairs \((j,s)\) lexicographically, namely
\[
    (j,s)<(j',s')
\]
if either \(j<j'\), or \(j=j'\) and \(s<s'\).

Since \(\omega_{j,s}\in H_{j-1}\) but its class in
\(H_{j-1}/H_j\) is non-zero, one has
\begin{equation}\label{eq:case3-vanishing}
    1
    \leq
    \operatorname{ord}_{p_j}(\omega_{j,s})
    <
    a_j,
    \qquad
    s=1,\ldots,d_j.
\end{equation}
In particular,
\[
    (2g-2)m_j
    -
    d\,\operatorname{ord}_{p_j}(\omega_{j,s})
    \neq0.
\]
Finally, choose a basis \(\mathcal B_n\) of \(H_n\).
The union
\[
    \mathcal B_0
    :=
    \bigcup_{j=1}^n
    \{\omega_{j,1},\ldots,\omega_{j,d_j}\}
    \cup
    \mathcal B_n
\]
is a basis of \(H_0\).
Indeed, this follows successively by lifting bases of the quotients
\(H_{j-1}/H_j\) along the filtration.
Observe that,
since $\deg\left(\sum_{i=1}^n a_i p_i\right)=2g-2,$
one has \(\dim H_n\leq1\). Moreover,
\begin{equation}
    H_n\neq0 \Leftrightarrow \omega_C
    \simeq
    \mathcal O_C\left(\sum_{i=1}^n a_i p_i\right) \Leftrightarrow [C;p_1,\ldots,p_n]
    \in\mathcal P(a_1,\ldots,a_n).
\end{equation}

In this case \(H_n\) is generated by a differential
\(\omega_{\mathrm{ab}}\) satisfying
\[
    \operatorname{div}(\omega_{\mathrm{ab}})
    =
    \sum_{i=1}^n a_i p_i.
\]

\subsection{Associated Schiffer deformations}

We now associate infinitesimal deformations of the pointed curve to the
adapted bases constructed above. Recall that the first \(r\) elements of
the basis satisfy \eqref{eq:triangular-evaluation}.

For every \(i=1,\ldots,r\), since \(\omega_i(p_i)\neq0\), choose the
unique vector $v_i\in T_{p_i}C$
such that $\omega_i(p_i)(v_i)=1.$
Through the natural inclusion
\[
    T_{p_i}C
    \hookrightarrow
    \bigoplus_{j=1}^n T_{p_j}C
    \hookrightarrow
    H^1(C,T_C(-E)),
\]
we regard \(v_i\) also as an infinitesimal deformation of the pointed
curve, and use the same notation for its image.
As explained in Remark~\ref{rem:vertical-schiffer}, the class \(v_i\)
represents the infinitesimal variation of the marked point \(p_i\), with
the complex structure of \(C\) fixed, and we refer to it as the
vertical Schiffer deformation associated with \(\omega_i\). In
particular,
\[
    dq(v_i)=0
    \qquad\text{in } H^1(C,T_C).
\]

Equivalently, if \(\rho_i\) is a smooth bump function centered at
\(p_i\), with support contained in a neighbourhood on which
\(\omega_i\) does not vanish, then \(v_i\) is represented in the
Dolbeault model by $v_i
    =
    \left[
        \frac{\bar\partial\rho_i}{\omega_i}
    \right].$

For the remaining elements of the adapted basis we consider the
corresponding higher Schiffer deformations. 
In Case~1, let \(p_s\) be the marked point with \(m_s<0\) chosen in
the construction of the basis. For \(i=r+1,\ldots,g\), let
\(\rho_s\) be a smooth bump function centered at \(p_s\) and set
\[
    \eta_i
    :=
    \left[
        \frac{\bar\partial\rho_s}{\omega_i}
    \right]
    \in H^1(C,T_C(-E)).
\]
In Case~2, let \(p_s\) be the marked point for which
\(a_s\notin\mathbb Z\). For \(i=r+1,\ldots,g\), define
\[
    \eta_i
    :=
    \left[
        \frac{\bar\partial\rho_s}{\omega_i}
    \right]
    \in H^1(C,T_C(-E)),
\]
where \(\rho_s\) is a smooth bump function centered at \(p_s\).
In Case~3, for every lifted differential
\[
    \omega_{j,s},
    \qquad
    s=1,\ldots,d_j,
\]
choose a smooth bump function \(\rho_j\) centered at \(p_j\) and define
\[
    \eta_{j,s}
    :=
    \left[
        \frac{\bar\partial\rho_j}{\omega_{j,s}}
    \right]
    \in H^1(C,T_C(-E)).
\]
The ordering of the adapted bases is chosen precisely so that these
deformations have a triangular behaviour with respect to the
corresponding holomorphic differentials.

\subsection{Proof of the Main Theorem}

\begin{proof}
Fix \(y\in Y\) and set
\[
    C:=C_y,
    \qquad
    p_i:=p_i(y),
    \qquad
    D:=\sum_{i=1}^n m_i p_i,
    \qquad
    E:=p_1+\cdots+p_n.
\]
Via the étale modular map, we identify $T_yY
    \simeq
    H^1(C,T_C(-E))$. We first prove part~\((2)\) of the theorem. Thus we assume that
\([C;p_1,\ldots,p_n]\) does not satisfy the conditions appearing in
part~\((1)\).

Let
\[
    r
    =
    \operatorname{codim}_{H^0(C,\omega_C)}
    H^0(C,\omega_C(-E)).
\]
Recall that the first \(r\) elements of the adapted basis satisfy
\begin{equation}\label{eq:proof-triangular-vertical}
    \omega_j(p_i)=0
    \quad\text{for }i<j,
    \qquad
    \omega_j(p_j)\neq0,
    \quad j=1,\ldots,r.
\end{equation}
Let $v_1,\ldots,v_r
    \in H^1(C,T_C(-E))$
be the corresponding vertical Schiffer deformations, normalized so that $\omega_i(p_i)(v_i)=1.$

If \(r<g\), let moreover
\[
    \{\eta_\lambda\}_{\lambda\in\Lambda}
\]
denote the higher Schiffer deformations constructed from the remaining
elements of the adapted basis. Here, in Cases~1 and~2, the index set
\(\Lambda\) is simply the ordered set of the remaining differentials,
whereas in Case~3 we use the indices
\[
    \lambda=(j,s),
    \qquad
    \omega_\lambda=\omega_{j,s},
    \qquad
    \eta_\lambda=\eta_{j,s},
\]
ordered as in the construction of the basis. In all cases the
construction gives $|\Lambda|=g-r.$
If \(r=g\), we simply take \(\Lambda=\varnothing\).

We now prove the following claim:
\begin{equation}\label{eq:main-claim-proof}
\begin{split}
    &\xi
    =
    \sum_{i=1}^r\alpha_i v_i
    +
    \sum_{\lambda\in\Lambda}\beta_\lambda\eta_\lambda
    \in \mathcal N_y^\nu \Longrightarrow \,\alpha_i=0, \,
    \beta_\lambda=0 \quad\forall \, i,\lambda
\end{split}
\end{equation}



We first show that all the coefficients \(\beta_\lambda\) vanish. Assume that some \(\beta_\lambda\neq0\), and let
\(\lambda_0\) be maximal, with respect to the ordering fixed above,
among the indices such that
\[
    \beta_{\lambda_0}\neq0.
\]
Write
\[
    \omega:=\omega_{\lambda_0}, \quad p:=p_{\lambda_0},
\]
for the corresponding holomorphic differential and the center of the
Schiffer deformation.

By construction of the adapted basis, the higher Schiffer
deformations satisfy the triangular property
\[
    dq(\eta_\lambda)\cdot\omega=0
    \qquad
    \text{for every }\lambda\leq\lambda_0.
\]
Moreover, $dq(v_i)=0$ for $i=1,\ldots,r.$
Since the coefficients corresponding to indices
\(\lambda>\lambda_0\) vanish by the choice of \(\lambda_0\), it follows
that
\[
    dq(\xi)\cdot\omega=0.
\]
Hence \(\delta\nu(y)\) can be evaluated on
\(\omega\otimes\xi\). Since \(\xi\in \mathcal N_y^\nu\), by \eqref{eq:xi in N} and Remark \ref{rem:loc-constant-delta} we have
\begin{equation}\label{eq:delta-zero-proof}
    \delta\nu(y)(\omega\otimes\xi)=0.
\end{equation}

Set $Z_\omega:=\operatorname{div}(\omega).$
We represent the value of the normal function by
\[
    \nu(y)
    =
    \mathcal O_C\bigl((2g-2)D-dZ_\omega\bigr).
\]
Since every differential involved in the higher Schiffer construction
belongs to \(H^0(C,\omega_C(-E))\), all the marked points are zeros of
\(\omega\). Thus the first term in the Griffiths formula vanishes.

For each \(\lambda\leq\lambda_0\), the quotient $\frac{\omega}{\omega_\lambda}$
is holomorphic in a neighbourhood of the center of \(\eta_\lambda\).
Consequently, a solution of the Dolbeault equation entering the
Griffiths formula (cf. Theorem \ref{Griff_theorem}) is given locally by
\[
    h
    =
    \sum_{\lambda\leq\lambda_0}
    \beta_\lambda
    \rho_\lambda
    \frac{\omega}{\omega_\lambda}.
\]
The ordering of the vanishing orders implies that all the preceding
terms vanish at the center \(p\), while the term corresponding to
\(\lambda_0\) equals \(\beta_{\lambda_0}\). Hence
\[
    h(p)=\beta_{\lambda_0},
\]
and \(h\) vanishes at the other points of the support relevant to the
Griffiths formula. Therefore \eqref{eq:delta-zero-proof} gives
\begin{equation}\label{eq:diagonal-coefficient}
    0
    =
    \beta_{\lambda_0}
    \left(
        (2g-2)m_p
        -
        d\,\operatorname{ord}_{p}(\omega)
    \right),
\end{equation}
where \(m_p\) denotes the coefficient of \(p\) in \(D\).

By construction of the adapted basis,
\[
    (2g-2)m_p
    -
    d\,\operatorname{ord}_{p}(\omega)
    \neq0.
\]
Thus $\beta_{\lambda_0}=0,$
contrary to the choice of \(\lambda_0\). We conclude that
\[
    \beta_\lambda=0
    \qquad
    \text{for every }\lambda\in\Lambda.
\]

It remains to consider
\[
    \xi=\sum_{i=1}^r\alpha_i v_i.
\]
Since the \(v_i\) are vertical deformations, $dq(\xi)=0,$
and the formula \eqref{formula-vertical-inf-invariant} for the infinitesimal invariant along marked
deformations gives
\[
    \delta\nu(y)(\omega\otimes\xi)
    =
    (2g-2)
    \sum_{i=1}^r
    \alpha_i m_i\,\omega(p_i)(v_i)
\]
for every \(\omega\in H^0(C,\omega_C)\).

We now evaluate successively against $\omega_r,\omega_{r-1},\ldots,\omega_1.$
Since \(\xi\in \mathcal N_y^\nu\), evaluating first against \(\omega_r\) and
using \eqref{eq:proof-triangular-vertical} gives
\[
    0
    =
    (2g-2)\alpha_r m_r
    \omega_r(p_r)(v_r).
\]
By construction, $\omega_r(p_r)(v_r)=1,$
and \(m_r\neq0\) by hypothesis. Hence $\alpha_r=0.$
Assume inductively that $\alpha_{j+1}=\cdots=\alpha_r=0.$
Evaluating against \(\omega_j\), the conditions $\omega_j(p_i)=0$ for $i<j$
give
\[
    0
    =
    (2g-2)\alpha_jm_j
    \omega_j(p_j)(v_j).
\]
Since \(m_j\neq0\) by hypothesis and \(\omega_j(p_j)(v_j)=1\), we obtain \(\alpha_j=0\).
Proceeding by reverse induction, we obtain
\[
    \alpha_1=\cdots=\alpha_r=0.
\]
This proves \eqref{eq:main-claim-proof}.
We now draw its two consequences. First, the deformations $v_1,\ldots,v_r,\eta_\lambda$ for $\lambda\in\Lambda$
are linearly independent. Hence, let
\[
    V
    :=
    \left\langle
        v_1,\ldots,v_r,
        \eta_\lambda\ ;\ \lambda\in\Lambda
    \right\rangle.
\]
Since $r+|\Lambda|=g,$ we have $\dim V=g.$
Second, the same claim gives $V\cap \mathcal N_y^\nu=\{0\}.$


Then, we obtain
\[
    \operatorname{codim}\mathcal N_y^\nu\geq g.
\]
Together with the fact that $\operatorname{rk}_y(\nu) \leq g$ \eqref{codim < g},
gives $\operatorname{codim}\mathcal N_y^\nu=g.$

We now prove part~\((1)\). Assume that $a_i=\frac{(2g-2)m_i}{d}$ are positive integers and that
\[
    [C;p_1,\ldots,p_n]
    \in
    \mathcal P(a_1,\ldots,a_n).
\]
We are in Case~3 of the construction. Recall that $H_n
    =
    H^0\left(
        C,
        \omega_C
        \left(-\sum_{i=1}^na_ip_i\right)
    \right)$
is one-dimensional and is generated by a holomorphic differential
\(\omega_{\mathrm{ab}}\) satisfying
\[
    \operatorname{div}(\omega_{\mathrm{ab}})
    =
    \sum_{i=1}^na_ip_i.
\]
Thus the lifted bases of the quotients
\(H_{j-1}/H_j\), together with the \(r\) vertical deformations, provide
exactly \(g-1\) Schiffer deformations; the remaining one-dimensional
space \(H_n\) corresponds precisely to the Abelian differential
\(\omega_{\mathrm{ab}}\).

Let \(V_{\mathrm{ab}}\subset T_yY\) be the span of these \(g-1\)
deformations. The same argument used above, applied only to the lifted
differentials of the quotients and to the vertical deformations, gives
\[
    V_{\mathrm{ab}}\cap \mathcal N_y^\nu=\{0\}.
\]
Indeed, for every lifted differential the coefficient $(2g-2)m_j
    -
    d\,\operatorname{ord}_{p_j}(\omega_{j,s})$
is non-zero, since $\operatorname{ord}_{p_j}(\omega_{j,s})<a_j.$
The only differential for which the corresponding coefficient may
vanish is \(\omega_{\mathrm{ab}}\), for which $\operatorname{ord}_{p_j}(\omega_{\mathrm{ab}})=a_j.$
This is precisely the one direction omitted from the construction.

It follows that $\dim V_{\mathrm{ab}}=g-1$
and hence
\[
    \operatorname{codim}\mathcal N_y^\nu\geq g-1.
\]
Then , by
\eqref{disug g-1}, we have $\operatorname{codim}\mathcal N_y^\nu=g-1,$
which concludes the proof.
\end{proof}

\paragraph{Acknowledgments}
I would like to thank my PhD advisor Gian Pietro Pirola for sharing many of his ideas and insights on
this problem. The author is member of GNSAGA (INdAM).

\newpage
\printbibliography

\end{document}